\documentclass[reqno,12pt]{amsart}
\usepackage{amscd,amsfonts,amssymb}
\usepackage[T2A]{fontenc}
\usepackage[cp1251]{inputenc}
\numberwithin{equation}{section}
\newtheorem{Theorem}{Theorem}[section]

\theoremstyle{definition}

\theoremstyle{remark}
\newtheorem{Remark}{Remark}[section]

\author{A.\,A. Kon'kov}
\address{Department of Differential Equations,
Faculty of Mechanics and Mathematics,
Mo\-s\-cow Lo\-mo\-no\-sov State University,
Vorobyovy Gory,
Moscow, 119992 Russia.
}
\email{konkov@mech.math.msu.su}
\author{A.\,E. Shishkov}
\address{
Center of Nonlinear Problems of Mathematical Physics,
RUDN University,
Mi\-klu\-k\-ho-Mak\-la\-ya str. 6,
Moscow, 117198 Russia.
}
\email{aeshkv@yahoo.com}
\email{konkov@mech.math.msu.su}
\author{M.\,D. Surnachev}
\address{
M.V.Keldysh Institute of Applied Mathematics of the Russian Academy of Sciences,
Moscow, 125047 Russia.
}
\title[On the existence of global solutions]{On the existence of global solutions of second-order quasilinear elliptic inequalities}
\date{}

\begin{document}

\begin{abstract}
We study the existence of global positive solutions of the differential inequalities 
$$
	- \operatorname{div} A (x, u, \nabla u)
	\ge
	f (u)
	\quad
	\mbox{in } {\mathbb R}^n,
$$
where $n \ge 2$ and $A$ is a Carath\'eodory function such that
$$
	(A (x, s, \zeta) - A (x, s, \xi))(\zeta - \xi) \ge 0,
$$
$$
	C_1
	|\xi|^p
	\le
	\xi
	A (x, s, \xi),
	\quad
	|A (x, s, \xi)|
	\le
	C_2
	|\xi|^{p-1},
	\quad
	C_1, C_2 > 0,
	\;
	p > 1,
$$
for almost all
$x \in {\mathbb R}^n$
and for all $s \in {\mathbb R}$ and $\zeta, \xi \in {\mathbb R}^n$. 
\end{abstract}

\maketitle

\section{Introduction}

We consider the inequality
\begin{equation}
	- \operatorname{div} A (x, u, \nabla u)
	\ge
	f (u)
	\quad
	\mbox{in } {\mathbb R}^n,
	\label{1.1}
\end{equation}
where $n \ge 2$ and $A$ is a Carath\'eodory function such that
\begin{equation}
	(A (x, s, \zeta) - A (x, s, \xi))(\zeta - \xi) \ge 0,
	\label{1.2}
\end{equation}
\begin{equation}
	C_1
	|\xi|^p
	\le
	\xi
	A (x, s, \xi),
	\quad
	|A (x, s, \xi)|
	\le
	C_2
	|\xi|^{p-1}
	\label{1.3}
\end{equation}
for almost all
$x \in {\mathbb R}^n$
and for all $s \in {\mathbb R}$ and $\zeta, \xi \in {\mathbb R}^n$. 
Here $p > 1$, $C_1 > 0$, and $C_2 > 0$ are real numbers independent of $x$, $s$, $\zeta$, and $\xi$.

The function $f$ in the right-hand side of~\eqref{1.1} is assumed to be non-negative and non-decreasing on the interval $[0, \varepsilon]$ for some $\varepsilon \in (0, \infty)$.

As is customary, by $B_r^x$ and $S_r^x$ we denote the open ball and the sphere in ${\mathbb R}^n$ of radius $r > 0$ centered at $x \in {\mathbb R}^n$. In the case of $x = 0$, let us write $B_r$ and $S_r$ instead of $B_r^0$ and $S_r^0$, respectively.

A function $u \in W_{p, loc}^1 ({\mathbb R}^n)$ is called a solution of~\eqref{1.1} if $f (u) \in L_{1, loc} ({\mathbb R}^n)$ and
$$
	\int_{
		{\mathbb R}^n
	}
	A (x, u, \nabla u)
	\nabla \varphi
	\,
	dx
	\ge
	\int_{
		{\mathbb R}^n
	}
	f (u)
	\varphi
	\,
	dx
$$
for any non-negative function $\varphi \in C_0^\infty ({\mathbb R}^n)$.

We additionally assume that solutions of~\eqref{1.1} satisfy the condition
\begin{equation}
	\operatorname*{ess\,inf}\limits_{
		{\mathbb R}^n
	}
	u
	=
	0
	\label{1.6}
\end{equation}
This obviously implies that the function $u$ is non-negative almost everywhere in ${\mathbb R}^n$.

\pagebreak

Analogously, by a solution of the equation
$$
	- \operatorname{div} A (x, u, \nabla u)
	=
	g (x)
	\quad
	\mbox{in } {\mathbb R}^n,
$$
where $g \in L_{1, loc} ({\mathbb R}^n)$, we mean a function $u \in W_{p, loc}^1 ({\mathbb R}^n)$ such that
$$
	\int_{
		{\mathbb R}^n
	}
	A (x, u, \nabla u)
	\nabla \varphi
	\,
	dx
	=
	\int_{
		{\mathbb R}^n
	}
	g (x)
	\varphi
	\,
	dx
$$
for all $\varphi \in C_0^\infty ({\mathbb R}^n)$.

The absence of non-trivial solutions of differential equations and inequalities or, in other words, the blow-up phenomenon is traditionally attracted the attention of mathematicians~[1--13].
It is easy to see that for inequalities~\eqref{1.1} the only relevant case is $n > p$. 
Indeed, if $p \ge n$ and $u$ is a non-negative solution of the inequality
$$
	- \operatorname{div} A (x, u, \nabla u)
	\ge
	0
	\quad
	\mbox{in } {\mathbb R}^n,
$$
then $u$ is a constant function~\cite{MPbook}.
In paper~\cite{AGQ}, for the differential inequalities
\begin{equation}
	- \operatorname{div} (|\nabla u|^{p-2} \nabla u) \ge f (u)
	\quad
	\mbox{in } {\mathbb R}^n
	\label{1.4}
\end{equation}
which are a partial case of~\eqref{1.1}, it was shown that, in the case of
\begin{equation}
	\int_0^\varepsilon
	\frac{
		f (t)
		\,
		dt
	}{
		t^{1 + n (p - 1) / (n - p)}
	}
	=
	\infty,
	\label{1.5}
\end{equation}
any solution satisfying condition~\eqref{1.6}
is trivial, i.e. is equal to zero almost everywhere in ${\mathbb R}^n$.
In so doing, if~\eqref{1.5} is not fulfilled, a positive solution of~\eqref{1.4} was constructed that tends to zero as $x \to \infty$. 

A shortcoming of~\cite{AGQ} is that the arguments presented there make significant use of the spherical symmetry of the differential operator. The method proposed by the authors of~\cite{AGQ} is based on the reduction of arbitrary solutions of~\eqref{1.4} to spherically symmetric ones. This method is obviously not applicable to inequalities of the general form~\eqref{1.1}.
In~\cite{KSMZ}, by a completely different approach, it was shown that~\eqref{1.5} is a sufficient condition for the absence of nontrivial solutions in the case of inequalities~\eqref{1.1} as well. However, the question of whether this condition is necessary remained open. Theorem~\ref{T2.1} given below provides an answer to this question.

\section{Main results}

\begin{Theorem}\label{T2.1}
Let $n > p$ and, moreover,
\begin{equation}
	\int_0^\varepsilon
	\frac{
		f (t)
		\,
		dt
	}{
		t^{1 + n (p - 1) / (n - p)}
	}
	<
	\infty.
	\label{T2.1.1}
\end{equation}
Then inequality~\eqref{1.1} has a positive solution satisfying condition~\eqref{1.6}.
\end{Theorem}


\begin{Theorem}\label{T2.2}
Let $n > p$ and, moreover, $F : [0, \infty) \to [0, \infty)$ be a non-increasing function such that
\begin{equation}
	\int_0^\infty
	r^{n (p - 1) / p} F (r)
	\,
	dr
	<
	\infty.
	\label{T2.2.1}
\end{equation}
Then the equation
\begin{equation}
	- \operatorname{div} A (x, u, \nabla u)
	=
	F (|x|)
	\quad
	\mbox{in } {\mathbb R}^n
	\label{T2.2.2}
\end{equation}
has a solution satisfying condition~\eqref{1.6}.
\end{Theorem}

\begin{Remark}\label{R2.1}
If, under the assumptions of Theorem~\ref{T2.2}, the function $F$ does not vanish almost everywhere on the interval $[0, \infty)$, then any non-negative solution of~\eqref{T2.2.2} is positive almost everywhere in ${\mathbb R}^n$. Indeed, by the weak Harnack inequality, we obtain
$$
	\left(
		\frac{
			1
		}{
			\operatorname{mes} B_{2 r}
		}
		\int_{
			B_{2 r}
		}
		u^\lambda
		\,
		dx
	\right)^{1 / \lambda}
	\le
	C
	\operatorname*{ess\,inf}\limits_{
		B_r
	}
	u
$$
for all $\lambda \in (0, n (p - 1) / (n - p))$ and $r \in (0, \infty)$~\cite{Serrin, Trudinger, TW}.
\end{Remark}

\begin{proof}[Proof of Theorem~\ref{T2.2}]
Let us agree to denote by $C$ various positive constants that can depend only on $p$, $n$, and the ellipticity constants $C_1$ and $C_2$.
As in~\cite{Mazya}, by $L_p^1 ({\mathbb R}^n)$ we mean the set of distributions $u \in {\mathcal D}' ({\mathbb R}^n)$ such that
$\nabla u \in L_p ({\mathbb R}^n)$.
The set $L_p^1 ({\mathbb R}^n)$ is obviously a linear space with the seminorm
\begin{equation}
	\| u \|_{
		L_p^1 ({\mathbb R}^n)
	}
	=
	\left(
		\int_{
			{\mathbb R}^n
		}
		|\nabla u|^p
		\,
		dx
	\right)^{1 / p}.
	\label{PT2.2.1}
\end{equation}
It is well-known that
$L_p^1 ({\mathbb R}^n) \subset L_{p, loc} ({\mathbb R}^n)$~\cite[Sec. 1.1.2]{Mazya}.
By $V$ we denote the space of functions $u \in L_p^1 ({\mathbb R}^n)$ satisfying the condition
\begin{equation}
	\int_{
		{\mathbb R}^n
	}
	\frac{
		|u|^p
	}{
		|x|^p
	}
	\,
	dx
	<
	\infty.
	\label{PT2.2.2}
\end{equation}
It is easy to see that~\eqref{PT2.2.1} is a norm on $V$.
Indeed, let $\eta \in C_0^\infty (B_2)$ be a non-negative function equal to one on $B_1$. Assuming that $\eta$ is extended by zero to the set ${\mathbb R}^n \setminus B_2$, we put
$$
	\eta_R (x) = \eta
	\left(
		\frac{x}{R}
	\right),
	\quad
	x \in {\mathbb R}^n, 
	\quad
	R > 0.
$$
Let $u \in V$. From Hardy's inequality~\cite[ Sec. 2.1.6, Corollary 2]{Mazya}, it follows that
$$
	\int_{
		{\mathbb R}^n
	}
	\frac{
		|\eta_R u|^p
	}{
		|x|^p
	}
	\,
	dx
	\le
	C
	\int_{
		{\mathbb R}^n
	}
	|\nabla (\eta_R u)|^p
	\,
	dx,
$$
whence we immediately obtain
$$
	\int_{
		B_R
	}
	\frac{
		|u|^p
	}{
		|x|^p
	}
	\,
	dx
	\le
	C
	\left(
		\int_{
			B_{2 R}
		}
		|\nabla u|^p
		\,
		dx
		+
		\int_{
			B_{2 R} \setminus B_R
		}
		\frac{|u|^p}{R^p}
		\,
		dx
	\right)
$$
for all $R > 0$.
In view of~\eqref{PT2.2.2}, the second summand in the right-hand side of the last inequality tends to zero as $R \to \infty$.
Thus, passing to the limit as $R \to \infty$, we have
\begin{equation}
	\int_{
		{\mathbb R}^n
	}
	\frac{
		|u|^p
	}{
		|x|^p
	}
	\,
	dx
	\le
	C
	\int_{
		{\mathbb R}^n
	}
	|\nabla u|^p
	\,
	dx.
	\label{PT2.2.3}
\end{equation}
With respect to the norm~\eqref{PT2.2.1} the space $V$ is obviously complete.
In so doing, the mapping $u \mapsto \nabla u$ isometrically embeds $V$ into
$$
	(L_p ({\mathbb R}^n))^n
	=
	\underbrace{L_p ({\mathbb R}^n) \times \ldots \times L_p ({\mathbb R}^n)}_{n}.
$$
Since $(L_p ({\mathbb R}^n))^n$ is a separable uniformly convex Banach space, $V$ is also separable and uniformly convex, so $V$ is reflexive by Milman's theorem~\cite[Chapter 5, Sec. 2, Theorem 2]{Yosida}.
This, in particular, implies the separability of $V^*$.

We consider the operator ${\mathcal A} : V \to V^*$ defined by
$$
	({\mathcal A} u, v)
	=
	\int_{
		{\mathbb R}^n
	}
	A (x, u, \nabla u) \nabla v
	\,
	dx,
	\quad
	u, v \in V.
$$
From~\eqref{1.3}, it follows that
\begin{equation}
	\lim_{
		\| u \|_{
			L_p^1 ({\mathbb R}^n)
		}
		\to 
		\infty
	}
	\frac{
		({\mathcal A} u, u)
	}{
		\| u \|_{
			L_p^1 ({\mathbb R}^n)
		}
	}
	=
	\infty.
	\label{PT2.2.4}
\end{equation}
Moreover, ${\mathcal A}$ is a bounded operator, in other words, it maps any bounded set in $V$ to a set bounded in $V^*$.
It is also easy to verify that for all $v \in V$ the mapping
\begin{equation}
	u \mapsto ({\mathcal A} u, v)
	\label{PT2.2.14}
\end{equation}
is a continuous functional on any finite-dimensional subspace of $V$.
Indeed, let $W$ be an $m$-dimensional subspace of $V$ with the basis $w_i$, $i = 1, \ldots, m$, and let $v \in V$ be some function.
Since any two norms on $W$ are equivalent, it suffices to show that the mapping
$$
	(\lambda_1, \ldots, \lambda_m)
	\mapsto
	({\mathcal A} (\lambda_1 w_1 + \ldots + \lambda_m w_m), v) 
$$
is continuous in ${\mathbb R}^m$. Let $(\lambda_{1 k}, \ldots, \lambda_{m k}) \to (\lambda_1, \ldots, \lambda_m)$ as $k \to \infty$.
The function $A$ on the left in~\eqref{1.1} is Carath\'eodoric. This means that the mapping
$
	x \mapsto A (x, s, \xi)
$
is measurable for all $s \in {\mathbb R}$ and $\xi \in {\mathbb R}^n$ and the mapping
$
	(s, \xi) \mapsto A (x, s, \xi)
$ 
is continuous for almost all $x \in {\mathbb R}^n$. Hence,
$$
	\lim_{k \to \infty}
	A \left(
		x, 
		\sum_{i = 1}^m
		\lambda_{i k} w_i,
		\sum_{i = 1}^m
		\lambda_{i k} \nabla w_i
	\right) 
	\nabla v
	=
	A \left(
		x, 
		\sum_{i = 1}^m
		\lambda_i w_i,
		\sum_{i = 1}^m
		\lambda_i \nabla w_i
	\right) 
	\nabla v
$$
for almost all $x \in {\mathbb R}^n$.
At the same time, taking into account~\eqref{1.3}, we have
$$
	\left|
		A \left(
			x, 
			\sum_{i = 1}^m
			\lambda_{i k} w_i,
			\sum_{i = 1}^m
			\lambda_{i k} \nabla w_i
		\right) 
		\nabla v
	\right|
	\le
	C
	\left(
			\sum_{i = 1}^m
			|\lambda_{i k}| |\nabla w_i|
			+
			|\nabla v| 
	\right)^p,
	\quad
	k = 1,2,\ldots,
$$
for almost all $x \in {\mathbb R}^n$.
Thus, from Lebesgue's dominated convergence theorem, it follows that
$$
	\lim_{k \to \infty}
	\int_{
		{\mathbb R}^n
	}
	A \left(
		x, 
		\sum_{i = 1}^m
		\lambda_{i k} w_i,
		\sum_{i = 1}^m
		\lambda_{i k} \nabla w_i
	\right) 
	\nabla v
	\,
	dx
	=
	\int_{
		{\mathbb R}^n
	}
	A \left(
		x, 
		\sum_{i = 1}^m
		\lambda_i w_i,
		\sum_{i = 1}^m
		\lambda_i \nabla w_i
	\right) 
	\nabla v
	\,
	dx
$$
or, in other words,
$$
	\lim_{k \to \infty}
	({\mathcal A} (\lambda_{1 k} w_1 + \ldots + \lambda_{m k} w_m), v) 
	=
	({\mathcal A} (\lambda_1 w_1 + \ldots + \lambda_m w_m), v).
$$

Unfortunately, we cannot claim that ${\mathcal A}$ is a monotone operator, i.e. such an operator that
$$
	({\mathcal A} u - {\mathcal A} v, u - v) \ge 0
$$
for all $u, v \in V$. 
Nevertheless, by arguments similar to~\cite[Chapter~2, Theorem~2.1]{Lions}, it is easy to show that ${\mathcal A} : V \to V^*$ is a surjective mapping.
To avoid being unfounded, we present these arguments in full.

Let ${\mathcal F}$ be a continuous linear functional on $V$. We show that
\begin{equation}
	{\mathcal A} u = {\mathcal F}.
	\label{PT2.2.9}
\end{equation}
for some $u \in V$. We take a family of linearly independent functions $w_i \in V$, $i = 1,2,\ldots$, such that the closure of their linear span coincides with $V$. We denote by $V_m$ the linear subspace of $V$ generated by the functions $w_i$, $i = 1, \ldots, m$, with the scalar product
$$	
	\langle w_i, w_j \rangle
	=
	\delta_{ij},
$$
where $\delta_{ij}$ is the Kronecker delta.
There obviously exist a function ${\mathbb F} \in V_m$ and an operator ${\mathbb A} : V_m \to V_m$ such that
$$
	\langle {\mathbb A} u, v \rangle
	=
	({\mathcal A} u, v )
	\quad
	\mbox{and}
	\quad
	\langle {\mathbb F}, v \rangle
	=
	({\mathcal F}, v)
$$
for all $u, v \in V_m$.
Since
$$
	{\mathbb A} u
	=
	\sum_{i=1}^m
	({\mathcal A} u, w_i)
	w_i
$$
for all $u \in V_m$ and, moreover, the mappings $u \mapsto ({\mathcal A} u, w_i)$, $i = 1, \ldots, m$, are continuous on the finite-dimensional space $V_m$, the operator ${\mathbb A} : V_m \to V_m$ is continuous.
In view of~\eqref{PT2.2.4}, there exists a real number $R > 0$ such that
$$
	\langle {\mathbb A} u - {\mathbb F}, u \rangle
	\ge
	0
$$
for all $u \in V_m$ satisfying the condition $\langle u, u \rangle = R^2$.
Therefore, by~\cite[Chapter~1, Lemma~4.3]{Lions}, there exists $u_m \in V_m$ such that
$$
	{\mathbb A} u_m = {\mathbb F}
$$
or, in other words,
\begin{equation}
	({\mathcal A} u_m, v)
	=
	({\mathcal F}, v)
	\label{PT2.2.5}
\end{equation}
for all $v \in V_m$.
In particular,
\begin{equation}
	({\mathcal A} u_m, u_m)
	=
	({\mathcal F}, u_m),
	\label{PT2.2.6}
\end{equation}
whence in accordance with~\eqref{PT2.2.4} it follows that the sequence $\{ u_m \}_{m=1}^\infty$ is bounded in $V$. 
Since ${\mathcal A} : V \to V^*$ is a bounded operator, the sequence $\{ {\mathcal A} u_m \}_{m=1}^\infty$ is bounded in $V^*$.
Let us extract from $\{ u_m \}_{m=1}^\infty$ a subsequence that converges weakly in $V$, strongly in $L_{p, loc} ({\mathbb R}^n)$, and almost everywhere in ${\mathbb R}^n$ to some function $u \in V$. Such a subsequence obviously exists. In order to avoid cluttering the indices, we also denote it by $\{ u_m \}_{m=1}^\infty$.
Passing to the limit in~\eqref{PT2.2.6}, we obtain
\begin{equation}
	\lim_{m \to \infty}
	({\mathcal A} u_m, u_m)
	=
	({\mathcal F}, u).
	\label{PT2.2.7}
\end{equation}
In its turn, from~\eqref{PT2.2.5}, it follows that
$$
	\lim_{m \to \infty}
	({\mathcal A} u_m, w_i)
	=
	({\mathcal F}, w_i)
$$
for all $i = 1,2,\ldots$. 
Since the linear span of the functions $w_i$, $i = 1,2,\ldots$, is dense in $V$ and the sequence $\{ {\mathcal A} u_m \}_{m=1}^\infty$ is bounded in $V^*$, this yields
\begin{equation}
	\lim_{m \to \infty}
	({\mathcal A} u_m, v)
	=
	({\mathcal F}, v)
	\label{PT2.2.8}
\end{equation}
for all $v \in V$.

Let $v \in V$ be some function. Taking into account~\eqref{PT2.2.7}, \eqref{PT2.2.8}, and the fact that the sequence $\{ u_m \}_{m=1}^\infty$ weakly converges to $u$, we have
\begin{equation}
	\lim_{m \to \infty}
	({\mathcal A} u_m - {\mathcal A} v, u_m - v)
	=
	({\mathcal F} - {\mathcal A} v, u - v).
	\label{PT2.2.10}
\end{equation}
At the same time, from~\eqref{1.2}, it follows that
\begin{align}
	({\mathcal A} u_m - {\mathcal A} v, u_m - v)
	=
	{}
	&
	\int_{
		{\mathbb R}^n
	}
	(A (x, u_m, \nabla u_m) - A (x, v, \nabla v))
	(\nabla u_m - \nabla v)
	\,
	dx
	\nonumber
	\\
	=
	{}
	&
	\int_{
		{\mathbb R}^n
	}
	(A (x, u_m, \nabla u_m) - A (x, u_m, \nabla v))
	(\nabla u_m - \nabla v)
	\,
	dx
	\nonumber
	\\
	&
	{}
	+
	\int_{
		{\mathbb R}^n
	}
	(A (x, u_m, \nabla v) - A (x, v, \nabla v))
	(\nabla u_m - \nabla v)
	\,
	dx
	\nonumber
	\\
	{}
	\ge
	&
	\int_{
		{\mathbb R}^n
	}
	(A (x, u_m, \nabla v) - A (x, v, \nabla v))
	(\nabla u_m - \nabla v)
	\,
	dx.
	\label{PT2.2.11}
\end{align}
In so doing, the relation
\begin{align}
	&
	\lim_{m \to \infty}
	\int_{
		{\mathbb R}^n
	}
	(A (x, u_m, \nabla v) - A (x, v, \nabla v))
	(\nabla u_m - \nabla v)
	\,
	dx
	\nonumber
	\\
	&
	\quad
	{}
	=
	\int_{
		{\mathbb R}^n
	}
	(A (x, u, \nabla v) - A (x, v, \nabla v))
	(\nabla u - \nabla v)
	\,
	dx
	\label{PT2.2.12}
\end{align}
is valid.
To prove~\eqref{PT2.2.12}, it suffices to note that
\begin{align*}
	&
	\int_{
		{\mathbb R}^n
	}
	(A (x, u_m, \nabla v) - A (x, v, \nabla v))
	(\nabla u_m - \nabla v)
	\,
	dx
	\\
	&
	\quad
	{}
	=
	\int_{
		{\mathbb R}^n
	}
	(A (x, u, \nabla v) - A (x, v, \nabla v))
	(\nabla u - \nabla v)
	\,
	dx
	\\
	&
	\quad
	\phantom{{}=}
	{}
	+
	\int_{
		{\mathbb R}^n
	}
	(A (x, u, \nabla v) - A (x, v, \nabla v))
	(\nabla u_m - \nabla u)
	\,
	dx
	\\
	&
	\quad
	\phantom{{}=}
	{}
	+
	\int_{
		{\mathbb R}^n
	}
	(A (x, u_m, \nabla v) - A (x, u, \nabla v))
	(\nabla u_m - \nabla v)
	\,
	dx.
\end{align*}
Since $u_k \to u$ weakly in $V$ as $k \to \infty$, we obviously obtain
$$
	\lim_{m \to \infty}
	\int_{
		{\mathbb R}^n
	}
	(A (x, u, \nabla v) - A (x, v, \nabla v))
	(\nabla u_m - \nabla u)
	\,
	dx
	=
	0.
$$
Taking into account the estimate
\begin{align*}
	&
	\left|
		\int_{
			{\mathbb R}^n
		}
		(A (x, u_m, \nabla v) - A (x, u, \nabla v))
		(\nabla u_m - \nabla v)
		\,
		dx
	\right|
	\\
	&
	\quad
	{}
	\le
	\left(
		\int_{
			{\mathbb R}^n
		}
		|A (x, u_m, \nabla v) - A (x, u, \nabla v)|^{p / (p - 1)}
		\,
		dx
	\right)^{(p - 1) / p}
	\left(
		\int_{
			{\mathbb R}^n
		}
		|\nabla u_m - \nabla v|^p
		\,
		dx
	\right)^{1 / p},
\end{align*}
where the first factor in the right-hand side tends to zero as $k \to \infty$ by Lebesgue's dominated convergence theorem and the second factor does not exceed some constant independent of $m$, we also obtain
$$
	\lim_{m \to \infty}
	\int_{
		{\mathbb R}^n
	}
	(A (x, u_m, \nabla v) - A (x, u, \nabla v))
	(\nabla u_m - \nabla v)
	\,
	dx
	=
	0.
$$
Thus, combining~\eqref{PT2.2.10}, \eqref{PT2.2.11}, and~\eqref{PT2.2.12}, we arrive at the inequality
$$
	({\mathcal F} - {\mathcal A} v, u - v)
	\ge
	\int_{
		{\mathbb R}^n
	}
	(A (x, u, \nabla v) - A (x, v, \nabla v))
	(\nabla u - \nabla v)
	\,
	dx
$$
putting in which $v = u - \lambda w$, $w \in V$, $\lambda \in (0, \infty)$, we have
\begin{equation}
	({\mathcal F} - {\mathcal A} (u - \lambda w), w)
	\ge
	\int_{
		{\mathbb R}^n
	}
	(A (x, u, \nabla u - \lambda \nabla w) - A (x, u - \lambda w, \nabla u - \lambda \nabla w))
	\nabla w
	\,
	dx.
	\label{PT2.2.13}
\end{equation}
Since~\eqref{PT2.2.14} is a continuous mapping on any finite-dimensional subspace of $V$, one can assert that
$$
	\lim_{\lambda \to +0}
	({\mathcal F} - {\mathcal A} (u - \lambda w), w)
	=
	({\mathcal F} - {\mathcal A} u, w).
$$
On the other hand, by Lebesgue's dominated convergence theorem,
$$
	\lim_{\lambda \to +0}
	\int_{
		{\mathbb R}^n
	}
	(A (x, u, \nabla u - \lambda \nabla w) - A (x, u - \lambda w, \nabla u - \lambda \nabla w))
	\nabla w
	\,
	dx
	=
	0;
$$
therefore, passing in~\eqref{PT2.2.13} to the limit as $\lambda \to +0$, we obtain
$$
	({\mathcal F} - {\mathcal A} u, w) \ge 0.	
$$
Since $w$ can be an arbitrary function from the space $V$, this immediately leads to~\eqref{PT2.2.9}.

Now, let us show that the linear functional ${\mathcal F}$ defined by the formula
$$
	({\mathcal F}, u)
	=
	\int_{
		{\mathbb R}^n
	}
	F (|x|)
	u
	\,
	dx
$$
is continuous on $V$.
By Holder's inequality, the estimate
\begin{equation}
	\left|
		\int_{
			{\mathbb R}^n
		}
		F (|x|)
		u
		\,
		dx
	\right|
	\le
	\left(
		\int_{
			{\mathbb R}^n
		}
		\frac{
			|u|^p
		}{
			|x|^p
		}
		\,
		dx
	\right)^{1 / p}
	\left(
		\int_{
			{\mathbb R}^n
		}
		|x|^{p / (p - 1)}
		F^{p / (p - 1)} (|x|)
		\,
		dx
	\right)^{(p - 1) / p}
	\label{PT2.2.18}
\end{equation}
is valid for all $u \in V$.
As $F$ is a non-increasing function on the interval $[0, \infty)$, we have
\begin{align}
	\int_{
		{\mathbb R}^n
	}
	|x|^{p / (p - 1)}
	F^{{p / (p - 1)}} (|x|)
	\,
	dx
	&
	=
	|S_1|
	\int_0^\infty
	r^{{p / (p - 1) + n - 1}}
	F^{p / (p - 1)} (r)
	\,
	dr
	\nonumber
	\\
	&
	=
	|S_1|
	\sum_{i = - \infty}^\infty
	\int_{2^i}^{2^{i+1}}
	r^{{p / (p - 1) + n - 1}}
	F^{p / (p - 1)} (r)
	\,
	dr
	\nonumber
	\\
	&
	\le
	C
	\sum_{i = - \infty}^\infty
	2^{i (p / (p - 1) + n)} 
	F^{p / (p - 1)} (2^i).
\label{PT2.2.19}
\end{align}
where $|S_1|$ is the $(n-1)$-dimensional volume of the unit sphere in ${\mathbb R}^n$.
At the same time,
$$
	\int_0^\infty
	r^{n (p - 1) / p}
	F (r)
	\,
	dr
	=
	\sum_{i = - \infty}^\infty
	\int_{2^{i-1}}^{2^i}
	r^{n (p - 1) / p}
	F (r)
	\,
	dr
	\ge
	C
	\sum_{i = 1}^\infty
	2^{i (1 + n (p - 1) / p)}
	F (2^i),
$$
whence it follows that
$$
	\left(
		\int_0^\infty
		r^{n (p - 1) / p}
		F (r)
		\,
		dr
	\right)^{p / (p - 1)}
	\ge
	C
	\sum_{i = - \infty}^\infty
	2^{i (p / (p - 1) + n)} 
	F^{p / (p - 1)} (2^i).
$$
Combining the last inequality with~\eqref{PT2.2.18} and~\eqref{PT2.2.19}, we obtain
$$
	\left|
		\int_{
			{\mathbb R}^n
		}
		F (|x|)
		u
		\,
		dx
	\right|
	\le
	C
	\left(
		\int_{
			{\mathbb R}^n
		}
		\frac{
			|u|^p
		}{
			|x|^p
		}
		\,
		dx
	\right)^{1 / p}
	\int_0^\infty
	r^{n (p - 1) / p}
	F (r)
	\,
	dr,
$$
whence in accordance with~\eqref{PT2.2.3} it follows that
$$
	\left|
		\int_{
			{\mathbb R}^n
		}
		F (|x|)
		u
		\,
		dx
	\right|
	\le
	C
	\left(
		\int_{
			{\mathbb R}^n
		}
		|\nabla u|^p
		\,
		dx
	\right)^{1 / p}
	\int_0^\infty
	r^{n (p - 1) / p}
	F (r)
	\,
	dr
$$
for all $u \in V$. 

Thus, there exists a solution of~\eqref{PT2.2.9} which is obviously a solution of~\eqref{T2.2.2} belonging to the space $V$. Let us show that it satisfies condition~\eqref{1.6}. 
We denote
$$
	u_- (x)
	=
	\left\{
		\begin{aligned}
			&
			u (x),
			&
			u (x) < 0,
			\\
			&
			0,
			&
			u (x) \ge 0.
		\end{aligned}
	\right.
$$
It is easy to see that $u_- \in V$ and, in addition,
$$
	\nabla u_- (x)
	=
	\left\{
		\begin{aligned}
			&
			\nabla u (x),
			&
			u (x) < 0,
			\\
			&
			0,
			&
			u (x) \ge 0.
		\end{aligned}
	\right.
$$
Taking into account~\eqref{PT2.2.9}, we have
$$
	\int_{
		{\mathbb R}^n
	}
	A (x, \nabla u) \nabla u_-
	\,
	dx
	=
	\int_{
		{\mathbb R}^n
	}
	F (|x|) 
	u_-
	\,
	dx
	\le
	0.
$$
Since
$$
	\int_{
		{\mathbb R}^n
	}
	A (x, \nabla u) \nabla u_-
	\,
	dx
	=
	\int_{
		{\mathbb R}^n
	}
	A (x, \nabla u_-) \nabla u_-
	\,
	dx
$$
and, moreover,
$$
	C_1
	\int_{
		{\mathbb R}^n
	}
	|\nabla u_-|^p
	\,
	dx
	\le
	\int_{
		{\mathbb R}^n
	}
	A (x, \nabla u_-) \nabla u_-
	\,
	dx
$$
in accordance with~\eqref{1.3}, the function $u_-$ is a constant almost everywhere in ${\mathbb R}^n$. By~\eqref{PT2.2.2}, this constant can only be equal to zero. From the above, we can conclude that $u (x) \ge 0$ for almost all $x \in {\mathbb R}^n$. However, in this case, failure to satisfy condition~\eqref{1.6} also contradicts relation~\eqref{PT2.2.2}.
\end{proof}

\begin{proof}[Proof of Theorem~\ref{T2.1}]
We denote by $C$ various positive constants that can depend only on $p$, $n$, $\varepsilon$ and the ellipticity constants $C_1$ and $C_2$.

Without loss of generality, it can be assumed that the function $f$ on the right in~\eqref{1.1} is positive on the interval $(0, \varepsilon)$; otherwise we replace $f$ by
$$
	\tilde f (t)
	=
	\max
	\{
		f (t),
		t^{1 + n (p - 1) / (n - p)}
	\}.
$$

Let us put
$$
	F (r)
	=
	f
	\left(
		\varepsilon
		\left(
			1
			+
			\frac{r}{\delta}
		\right)^{- (n - p) / (p - 1)}
	\right),
$$
where the real number $\delta > 0$ will be chosen later.
Making the change of variables $t = \varepsilon (1 + r / \delta)^{- (n - p) / (p - 1)}$, we obtain
\begin{align}
	\int_0^\infty
	r^{n-1}
	F (r)
	\,
	dr
	&
	=
	\int_0^\infty
	r^{n - 1}
	f
	\left(
		\varepsilon
		\left(
			1
			+
			\frac{r}{\delta}
		\right)^{- (n - p) / (p - 1)}
	\right)
	dr
	\nonumber
	\\
	&
	=
	\frac{p - 1}{n - p}
	\frac{\delta^n}{\varepsilon}
	\int_0^\varepsilon
	\left(
		\left(
			\frac{\varepsilon}{t}
		\right)^{(p - 1) / (n - p)}
		-
		1
	\right)^{n - 1}
	\left(
		\frac{\varepsilon}{t}
	\right)^{(p - 1) / (n - p) + 1}
	f (t)
	\,
	dt
	\nonumber
	\\
	&
	\le
	\frac{p - 1}{n - p}
	\varepsilon^{n (p - 1) / (n - p)}
	\delta^n
	\int_0^\varepsilon
	\frac{
		f (t)
		\,
		dt
	}{
		t^{1 + n (p - 1) / (n - p)}
	},
	\label{PT2.1.1}
\end{align}
whence in accordance with~\eqref{T2.1.1} it follows that
\begin{equation}
	\int_0^\infty
	r^{n - 1}
	F (r)
	\,
	dr
	<
	\infty.
	\label{T2.2.1old}
\end{equation}
Since $n > p$, one can obviously assert that
$$
	n - 1 
	> 
	\frac{n (p - 1)}{p};
$$ 
therefore,~\eqref{T2.2.1old} implies~\eqref{T2.2.1}.
Thus, by Theorem~\ref{T2.2}, there exists a solution of~\eqref{T2.2.2} for which~\eqref{1.6} is valid. Let us show that this solution satisfies the estimate
\begin{equation}
	u (x)
	\le
	C
	\delta^{n / (p - 1)}
	|x|^{- (n - p) / (p - 1)}
	\left(
		\int_0^\varepsilon
		\frac{
			f (t)
			\,
			dt
		}{
			t^{1 + n (p - 1) / (n - p)}
		}
	\right)^{1 / (p - 1)}
	\label{PT2.1.2}
\end{equation}
for almost all $x \in {\mathbb R}^n$ and, moreover,
\begin{equation}
	\lim_{\delta \to +0}
	\operatorname*{ess\,sup}\limits_{
		{\mathbb R}^n
	}
	u
	=
	0.
	\label{PT2.1.3}
\end{equation}
Indeed, taking into account~\cite[Corollary~4.13]{KM}, we have
\begin{equation}
	u (x)
	\le
	C
	\int_0^\infty
	\left(
		\frac{
			\mu (B_r^x)
		}{
			r^{n - p}
		}
	\right)^{1 / (p - 1)}
	\frac{
		d r
	}{
		r
	}
	\label{PT2.1.4}
\end{equation}
for almost all $x \in {\mathbb R}^n$, where
$$
	\mu (\omega)
	=
	\int_\omega
	F (|x|)
	\,
	dx
$$
is a measure with the density $F (|x|)$ of a Lebesgue measurable set $\omega \subset {\mathbb R}^n$.

Representing~\eqref{PT2.1.4} in the form
\begin{equation}
	u (x)
	\le
	\int_0^{|x| / 2}
	\left(
		\frac{
			\mu (B_r^x)
		}{
			r^{n - p}
		}
	\right)^{1 / (p - 1)}
	\frac{
		d r
	}{
		r
	}
	+
	\int_{|x| / 2}^\infty
	\left(
		\frac{
			\mu (B_r^x)
		}{
			r^{n - p}
		}
	\right)^{1 / (p - 1)}
	\frac{
		d r
	}{
		r
	},
	\label{PT2.1.5}
\end{equation}
we estimate the first summand in the right-hand side of~\eqref{PT2.1.5}.
Since $F$ is a non-increasing function, the inequality
$$
	\mu (B_r^x)
	=
	\int_{B_r^x}
	F (|y|)
	\,
	dy
	\le
	|S_1|
	r^n
	F 
	\left( 
		\frac{|x|}{2} 
	\right)
$$
holds for all $r \in (0, |x| / 2)$. Consequently, 
$$
	\int_0^{|x| / 2}
	\left(
		\frac{
			\mu (B_r^x)
		}{
			r^{n - p}
		}
	\right)^{1 / (p - 1)}
	\frac{
		d r
	}{
		r
	}
	\le
	C
	|x|^{p / (p - 1)} 
	F^{1 / (p - 1)} 
	\left( 
		\frac{|x|}{2} 
	\right),
$$
whence in accordance with the fact that
$$
	F
	\left( 
		\frac{|x|}{2} 
	\right)
	\le
	n
	\left(
		\frac{|x|}{2}
	\right)^{- n}
	\int_0^{|x| / 2}
	r^{n - 1}
	F (r)
	\,
	dr	
$$
we obtain
$$
	\int_0^{|x| / 2}
	\left(
		\frac{
			\mu (B_r^x)
		}{
			r^{n - p}
		}
	\right)^{1 / (p - 1)}
	\frac{
		d r
	}{
		r
	}
	\le
	C
	|x|^{- (n - p) / (p - 1)}
	\left(
		\int_0^\infty
		r^{n - 1}
		F (r)
		\,
		dr	
	\right)^{1 / (p - 1)}
$$
for all $x \in {\mathbb R}^n \setminus \{ 0 \}$.
According to~\eqref{PT2.1.1}, this leads to the inequality
$$
	\int_0^{|x| / 2}
	\left(
		\frac{
			\mu (B_r^x)
		}{
			r^{n - p}
		}
	\right)^{1 / (p - 1)}
	\frac{
		d r
	}{
		r
	}
	\le
	C
	\delta^{n / (p - 1)}
	|x|^{- (n - p) / (p - 1)}
	\left(
		\int_0^\varepsilon
		\frac{
			f (t)
			\,
			dt
		}{
			t^{1 + n (p - 1) / (n - p)}
		}
	\right)^{1 / (p - 1)}
$$
for all $x \in {\mathbb R}^n \setminus \{ 0 \}$.

Now, we estimate the second summand in the right-hand side of~\eqref{PT2.1.5}. It is easy to see that
$$
	\int_{|x| / 2}^\infty
	\left(
		\frac{
			\mu (B_r^x)
		}{
			r^{n - p}
		}
	\right)^{1 / (p - 1)}
	\frac{
		d r
	}{
		r
	}
	\le
	C
	|x|^{- (n - p) / (p - 1)}
	\mu^{1 / (p - 1)} ({\mathbb R}^n)
$$
with
\begin{equation}
	\mu ({\mathbb R}^n)
	=
	|S_1|
	\int_0^\infty
	r^{n - 1}
	F (r)
	\,
	dr.
	\label{PT2.1.7}
\end{equation}
As above, by $|S_1|$ we denote the $(n - 1)$-dimensional volume of the unit sphere in ${\mathbb R}^n$.
Combining the last inequality with~\eqref{PT2.1.1}, we obtain
$$
	\int_{|x| / 2}^\infty
	\left(
		\frac{
			\mu (B_r^x)
		}{
			r^{n - p}
		}
	\right)^{1 / (p - 1)}
	\frac{
		d r
	}{
		r
	}
	\le
	C
	\delta^{n / (p - 1)}
	|x|^{- (n - p) / (p - 1)}
	\left(
		\int_0^\varepsilon
		\frac{
			f (t)
			\,
			dt
		}{
			t^{1 + n (p - 1) / (n - p)}
		}
	\right)^{1 / (p - 1)}
$$
for all $x \in {\mathbb R}^n \setminus \{ 0 \}$.
Hence,~\eqref{PT2.1.5} implies~\eqref{PT2.1.2}.

Let us prove relation~\eqref{PT2.1.3}. From~\eqref{PT2.1.4} it follows that
$$
	u (x)
	\le
	C
	\left(
		\int_0^\delta
		\left(
			\frac{
				\mu (B_r^x)
			}{
				r^{n - p}
			}
		\right)^{1 / (p - 1)}
		\frac{
			d r
		}{
			r
		}
		+
		\int_\delta^\infty
		\left(
			\frac{
				\mu (B_r^x)
			}{
				r^{n - p}
			}
		\right)^{1 / (p - 1)}
		\frac{
			d r
		}{
			r
		}
	\right)
$$
for almost all $x \in {\mathbb R}^n$. As $F(|x|) \le f (\varepsilon)$ for all $x \in {\mathbb R}^n$, we have
$$
	\mu (B_r^x)
	=
	\int_{B_r^x}
	F (|y|)
	\,
	dy
	\le
	|S_1|
	r^n
	f (\varepsilon)
$$
for all $x \in {\mathbb R}^n$ and $ r \in (0, \infty)$; therefore,
$$
	\int_0^\delta
	\left(
		\frac{
			\mu (B_r^x)
		}{
			r^{n - p}
		}
	\right)^{1 / (p - 1)}
	\frac{
		d r
	}{
		r
	}
	\le
	C 
	f^{1 /(p - 1)} (\varepsilon)
	\delta^{p / (p - 1)}
$$
for all $x \in {\mathbb R}^n$.
At the same time,
$$
	\int_\delta^\infty
	\left(
		\frac{
			\mu (B_r^x)
		}{
			r^{n - p}
		}
	\right)^{1 / (p - 1)}
	\frac{
		d r
	}{
		r
	}
	\le
	\frac{p - 1}{n - p}
	\delta^{- (n - p) / (p - 1)}
	\mu^{1 / (p - 1)} ({\mathbb R}^n)
$$
for all $x \in {\mathbb R}^n$,
whence in accordance with~\eqref{PT2.1.1} and~\eqref{PT2.1.7} in follows that
$$
	\int_\delta^\infty
	\left(
		\frac{
			\mu (B_r^x)
		}{
			r^{n - p}
		}
	\right)^{1 / (p - 1)}
	\frac{
		d r
	}{
		r
	}
	\le
	C
	\delta^{p / (p - 1)}
	\left(
		\int_0^\varepsilon
		\frac{
			f (t)
			\,
			dt
		}{
			t^{1 + n (p - 1) / (n - p)}
		}
	\right)^{1 / (p - 1)}.
$$
Thus,
$$
	u (x)
	\le
	C
	\delta^{p / (p - 1)}
	\left(
		f^{1 / (p - 1)} (\varepsilon)
		+
		\left(
			\int_0^\varepsilon
			\frac{
				f (t)
				\,
				dt
			}{
				t^{1 + n (p - 1) / (n - p)}
			}
		\right)^{1 / (p - 1)}
	\right)
$$
for almost all $x \in {\mathbb R}^n$, whence we immediately arrive at~\eqref{PT2.1.3}.

It is obvious that
$$
	2^{- (n - p) / (p - 1)}
	\le
	\left(
		1
		+
		\frac{|x|}{\delta}
	\right)^{- (n - p) / (p - 1)}.
$$
for all $x \in B_\delta$.
On the other hand, in view of~\eqref{PT2.1.3}, there exists a real number $\delta_1 > 0$ such that
$$
	u (x)
	\le 
	\varepsilon
	2^{- (n - p) / (p - 1)}
$$
for all $0 < \delta \le \delta_1$ and for almost all $x \in {\mathbb R}^n$. Hence, taking $0 < \delta \le \delta_1$, we obtain
\begin{equation}
	u (x)
	\le
	\varepsilon
	\left(
		1
		+
		\frac{|x|}{\delta}
	\right)^{- (n - p) / (p - 1)}
	\label{PT2.1.8}
\end{equation}
for almost all $x \in B_\delta$.

It is also easy to see that
$$
	2^{- (n - p) / (p - 1)}
	\delta^{(n - p) / (p - 1)}
	|x|^{- (n - p) / (p - 1)}
	\le
	\left(
		1
		+
		\frac{|x|}{\delta}
	\right)^{- (n - p) / (p - 1)}
$$
for all $x \in {\mathbb R}^n \setminus B_\delta$.
In so doing, according to~\eqref{PT2.1.2}, there exists a real number $\delta_2 > 0$ such that
$$
	u (x)
	\le
	\varepsilon
	2^{- (n - p) / (p - 1)}
	\delta^{(n - p) / (p - 1)}
	|x|^{- (n - p) / (p - 1)}
$$
for all $0 < \delta \le \delta_2$ and for almost all $x \in {\mathbb R}^n$.
Consequently, taking $0 < \delta \le \delta_2$, we get that~\eqref{PT2.1.8} holds for almost all $x \in {\mathbb R}^n \setminus B_\delta$.

Taking
$$
	\delta = \min \{ \delta_1, \delta_2 \},
$$ 
we obviously obtain, that~\eqref{PT2.1.8} is valid for almost all $x \in {\mathbb R}^n$.
Since $f$ is a non-decreasing function on the interval $[0, \varepsilon]$, this yields
$$
	f
	\left(
		\varepsilon
		\left(
			1
			+
			\frac{|x|}{\delta}
		\right)^{- (n - p) / (p - 1)}
	\right)
	\ge
	f (u)
$$
Thus, $u$ is a solution of inequality~\eqref{1.1}.

Finally, by Remark~\ref{R2.1}, the function $u$ is positive almost everywhere in ${\mathbb R}^n$. Indeed, at the beginning of the proof of Theorem~\ref{T2.1} we assumed, without loss of generality, that $f$ is a positive function on the interval $(0, \varepsilon)$. This immediately implies that $F (|x|)$ is positive for all $x \in {\mathbb R}^n$.

\end{proof}


\begin{thebibliography}{100}

\bibitem{MPbook}
E. Mitidieri, S.I. Pohozaev,
A priori estimates and blow-up of solutions to nonlinear partial
differential equations and inequalities,
Proc. V.A.~Steklov Inst. Math. {\bf 234} (2001), 3--383.

\bibitem{AGQ}
S. Alarc\'on, J. Garc\'ia-Meli\'an, A. Quaas,
Optimal Liouville theorems for supersolutions of elliptic equations with the Laplacian,
Ann. Sc. Norm. Super. Pisa Cl. Sci. 
{\bf 16}:5, (2016), 129--158.

\bibitem{AM}
L. D'Ambrosio, E. Mitidieri,
Entire solutions of quasilinear elliptic systems on Carnot groups,
Proc. V.A.~Steklov Inst. Math. {\bf 283} (2013), 3--19.

\bibitem{AMAdvMath}
L.~D'Ambrosio, E.~Mitidieri, 
A priori estimates, positivity results, and nonexistence theorems for quasilinear
degenerate elliptic inequalities,
Adv. Math. {\bf 224} (2010), 967--1020.

\bibitem{ASComPDE}
S. Armstrong, B. Sirakov,
Nonexistence of positive supersolutions of elliptic equations via the maximum principle,
Commun. Partial Diff. Eq. {\bf 36} (2011), 2011--2047.

\bibitem{GS}
E.I. Galakhov, O.A. Salieva,
Blow-up of solutions of some nonlinear inequalities with singularities on unbounded sets,
Math. Notes {\bf 98}:1-2 (2015), 222--229.

\bibitem{Keller}
J.B. Keller,
On solution of $\Delta u = f(u)$.
Comm. Pure. Appl. Math. {\bf 10} (1957), 503--510.

\bibitem{KV}
A.A. Kondratiev, L. V\'eron,
Asymptotic behavior of solutions of some nonlinear parabolic or elliptic equations. 
Asympt. Anal. {\bf 14} (1997), 117--156.

\bibitem{KSMZ}
A.A.~Kon'kov, A.E.~Shishkov,
On global solutions of second-order quasilinear elliptic inequalities,
Math. Notes, {\bf116}:5 (2024), 1014--1019.

\bibitem{Nonlinearity}
A.A. Kon'kov, A.E. Shishkov,
Generalization of the Keller-Osserman theorem for higher order differential inequalities,
Nonlinearity {\bf 32} (2019), 3012--3022.

\bibitem{JEEQ}
A.A. Kon'kov, A.E. Shishkov,
On blow-up conditions for nonlinear higher-order evolution inequalities,
J. Evol. Equ. {\bf 24} (2024), 1--27.

\bibitem{Korpusov}
M.O. Korpusov,
On the blow-up of the solution of an equation related to the Hamilton-Jacobi equation,
Math. Notes, {\bf 93}:1 (2013), 90--101.

\bibitem{Osserman}
R. Osserman,
On the inequality $\Delta u \ge f(u)$.
Pacific J. Math. {\bf 7} (1957), 1641--1647.

\bibitem{Serrin}
J. Serrin, 
Local behavior of solutions of quasi-linear equations, 
Acta Math. {\bf 111} (1964), 247--302.

\bibitem{Trudinger}
N.S. Trudinger, 
On Harnack type inequalities and their application to quasilinear elliptic equations, 
Commun. Pure Appl. Math. {\bf 20} (1967), 721--747.

\bibitem{TW}
N.S. Trudinger and X.-J. Wang, 
On the weak continuity of elliptic operators and applications to potential theory, 
Am. J. Math. {\bf 124} (2002), 369--410.

\bibitem{Mazya}
V.G. Maz'ya,
Sobolev spaces,
Leningrad. Gos. Univ.,
Leningrad, 1985 (Russian).
English transl., Springer Ser. Soviet Math., Springer-Verlag, Berlin 1985.

\bibitem{Yosida}
K. Yosida,
Functional analysis,
Springer--Verlag, Berlin--G\"ettingen--Heidelberg, 1965.

\bibitem{Lions}
J.l. Lions,
Quelques m\'ethodes de r\'esolution des probl\`emes aux limites non lin\'eaire,
Dunod Gauthier--Villars, Paris, 1969. 

\bibitem{KM}
T. Kilpel\"ainen, J.~Mal\'y,
The Wiener test and potential estimetes for quasilinear elliptic equations,
Acta Math. {\bf 172} (1994), 137--161.

\end{thebibliography}
\end{document}